\documentclass[11pt,article]{article}

\usepackage{amsmath}
\usepackage{amsthm}
  \usepackage{paralist}
  \usepackage{graphics} 
  \usepackage{epsfig} 
\usepackage{graphicx}
\usepackage{caption}
\usepackage{subcaption}
\usepackage{epstopdf}
 \usepackage[colorlinks=true]{hyperref}
 \usepackage{multirow}
\input{amssym.tex}

\newtheorem{theorem}{Theorem}[section]

\newtheorem{proposition}{Proposition}

 \numberwithin{equation}{section}
\newtheorem{remark}{Remark}

\newcommand{\keywords}

\def\bc{\begin{center}}       \def\ec{\end{center}}
\def\ba{\begin{array}}        \def\ea{\end{array}}
\def\be{\begin{equation}}     \def\ee{\end{equation}}
\def\bea{\begin{eqnarray}}    \def\eea{\end{eqnarray}}
\def\beaa{\begin{eqnarray*}}  \def\eeaa{\end{eqnarray*}}

\def\mathbb{\Bbb}

\begin{document}

\title{\bf Existence and stability of nonconstant positive steady states of morphogenesis models}
\author{Haohao Chen  \thanks{ {\tt chenhaohao@2013.swufe.edu.cn}.},  Bo Tong \thanks{ {\tt tongql@2013.swufe.edu.cn}.},  Qi Wang   \thanks{ {\tt qwang@swufe.edu.cn}.}\\
Department of Mathematics\\
Southwestern University of Finance and Economics\\
555 Liutai Ave, Wenjiang, Chengdu, Sichuan 611130, China
}

\date{}
\maketitle

\abstract
In this paper, We study an one--dimensional morphogenesis model considered by C. Stinner et al. in \cite{STW} (Math. Meth. Appl. Sci. 2012, \textbf{35} 445--465).  Under homogeneous boundary conditions, we prove the existence of nonconstant positive steady states through local bifurcation theories.  We also rigourously study the stability of the nonconstant solutions when the sensitivity function are chosen to be linear and logarithmic function respectively.  Finally, we present numerical solutions to illustrate the formation of stable spatially inhomogeneous patterns.  Our numerical simulations suggests that this model can develop very complicated and interesting structures even over one--dimensional finite domains.


\section{Introduction}

It is the goal of this paper to study the existence and stability of nonconstant positive steady states of the following one--dimensional system
\begin{equation}\label{11}
\left\{
\begin{array}{ll}
u_t=(d_1u_x-\chi u \Phi_x(v))_x+\lambda- u,& x \in (0,L),t>0,\\
v_t=d_2v_{xx}+1-(1+\lambda)v+u, & x\in(0,L),t>0,\\
u_x(x,t)=v_x(x,t)=0,&x=0,L,t>0,\\
u(x,0)=u_0(x)\geq0,v(x,0)=v_0(x)\geq0,&x \in (0,L),
\end{array}
\right.
\end{equation}
where $u=u(x,t)$, $v=v(x,t)$; $d_1$, $d_2$, $\chi$ and $\lambda$ are positive constants, and $\Phi$ is a continuously differentiable function.  This system was proposed by Bollenbach et al. in \cite{BKP} as a limit case of morphogenesis models that describe the transport of morphogens in Epithelia.  To be precise, $u(x,t)$ denotes the density of total ligand (an ion or molecule that binds to a central metal atom to form a metal complex) at space--time location $(x,t)$ and $v(x,t)$ represents the density of the free and bound receptor; $d_1$ and $d_2$ are the diffusion rates of the ligand and receptor respectively and $\chi$ measures the attraction rate of the receptor to the morphogen. $\Phi$ reflects the variation of morphogen transportation induced by the receptor concentration gradient with respect to the levels of receptor density.  Moreover, the boundary conditions represent an one--dimensional enclosed domain and the initial data are assumed to be nonnegative and not identically zero.

Morphogenesis (from the Greek \emph{morph¨º} shape and \emph{genesis} creation) is a biological process that has been studied since the early 20th century and it is an essential mechanism for the cellular growth and differentiation to create organs and their organizations.  Morphogenesis can take place in cell cultures, inside tumor cell masses or in a mature organism.  It causes organisms to develop their shapes and then controls the organized spatial distribution of cells during the embryonic development of an organism.  Morphogenetic responses usually occur through chemotaxis process which may be induced by hormones in organisms, environmental chemicals produced by other organisms or radionuclides released by the cells, or the mechanical stresses induced by spatial patterning of the cells.

Theoretical and mathematical modelings have been proposed to study how the locally secreted morphogens are transported to the cells of embryo over the past few decades.   The earliest ideas and mathematical explanations dates back to the seminal work Alan Turing [3] in 1952 on how physical processes affect biological growth and the development of natural patterns such as the spirals of phyllotaxis.  Turing correctly predicted that the diffusion and reaction of two different chemical signals, one activating and the other deactivating growth, were able to set up patterns of development.  Turing's revolutionary idea was that the diffusion, which is a smoothing process for single equation, can interact with chemical reactions to destabilize homogeneity for systems, therefore nontrivial patterns (nonconstant solutions) can emerge through bifurcations.  It is worthwhile to mention that chemotaxis is well accepted to be a leading organism in the spatio--temporal behaviors and pattern formations for many biological processes.  Recently, several authors have proposed and studied some PDEs to incorporate the chemotaxis in the transportation of morphogens in Epithelia.  See \cite{BKP}, \cite{MNS}, \cite{STW}, and the references therein.

The mathematical model proposed in \cite{BKP} consists of two equations and its general form in spatially one--dimensional domain reads
\begin{equation}\label{12}
\left\{
\begin{array}{ll}
u_t=(d_1(u,v)u_x-\chi (u,v)v_x)_x-k_1(u,v)u,&x \in (0,L),t>0,\\
v_t=d_2(u,v)v_{xx}+k_2(u,v)-k_3(u,v)v,& x\in(0,L),t>0,
\end{array}
\right.
\end{equation}
where $k_1$ and $k_3$ are the degradation rates of morphogen and receptor respectively, while $k_2$ measures the production of receptor at the cells surfaces.   In particular, the coefficients in \cite{BKP} are chosen to be
\[d_1(u,v)=\frac{d_1 v}{c_1u+c_2v}, d_2(u,v)=0, \chi(u,v)=\frac{d_2 v}{c_1u+c_2v}, k_3(u,v)=\frac{c_3v-c_4u}{v},\]
where $c_i$, $i=1,2,3,4$, are positive constants.  The degradation rate $k_1$ is assumed to be a constant and $k_2(u,v)$ is chosen to be a linear function of the receptor and the blind receptors.  The reader is referred to \cite{BKP} the mathematical modeling and the detailed justification for the choices of these functions.  In \cite{STW}, C. Stinner et al. investigated the well--posedness of the following simplified system
\begin{equation}\label{13}
\left\{
\begin{array}{ll}
u_t=(u_x-\chi \frac{u}{v}v_x)_x-\mu u,&x \in (0,L),t>0,\\
v_t=1-v+u,&x\in(0,L),t>0,
\end{array}
\right.
\end{equation}
subject to the boundary condition
\[u_x(0,t)-\chi \frac{u}{v}v_x(0,t)=-\nu,u_x(L,t)-\chi \frac{u}{v}v_x(L,t)=0,\]
with $\nu>0$ and nonnegative initial conditions $u(x,0)=u_0(x), v(x,0)=v_0(x),x\in(0,L)$.  It is proved that (\ref{13}) admits global and bounded weak solutions for all $\chi\in(0,1)$; moreover, positive solutions of (\ref{13}) converge to the unique steady state provided that $\chi$ is sufficiently small and $\lambda$ is sufficiently large.  We want to note that non--flux boundary condition is imposed at $x=1$ and the flux through $x=0$ is assumed to be related to the production rate of the ligand.

It is noteworthy to mention that the Keller--Segel type model (\ref{12}) has also been widely utilized in the modeling of chemotaxis.  The strongly coupled PDEs have demonstrated their abilities in capturing interesting and important phenomena through chemotaxis such as the traveling band and cellular aggregation.  The former can be modelled by travelling wave solutions and it is out of the scope of our attention.  The survey paper \cite{Wz} by Z. Wang is a good reference for the travelling wave solutions of the chemotaxis models.  To model the cellular aggregation phenomenon, there are two well--established methodologies available in literature.  One approach, first proposed by Childress and Perkus \cite{CP} in 1981, is to show that the time--dependent solution blows up in finite or infinite time with the $L^\infty$--norm of the solutions goes to infinity, then the aggregation is simulated by one $\delta$--function or a combination of several $\delta$--functions.  See \cite{N} and the survey paper \cite{HP} for works in this direction.  The $\delta$--function is evidently connected to the modeling of cellular aggregations, however, it is necessary to point out that, the singularity of $\delta$--function brings difficulties in modeling the cellular aggregations analytically and numerically.

Another approach is to show that the time--dependent system admits global--in--time solutions that converge to bounded steady states.  Then positive steady states with nontrivial structures such as spikes or transition layers can be employed to model the aggregation phenomenon.  For works and results in this direction, see the papers of Lin et al. \cite{LNT}, Ni and Takagi \cite{NT}, \cite{NT2} on the existence of boundary spike of the least--energy solutions through variational method, and the method initiated and developed by Wang \cite{CKWW,W,WX} on the formation of spikes through global bifurcation theory and the Helly's compactness theorem.

In this paper, we study a system general than (\ref{13}) in the following form,
\begin{equation}\label{14}
\left\{
\begin{array}{ll}
u_t=(d_1u_x-\chi u \Phi_x(v))_x+\lambda-\mu u,& x \in (0,L),t>0,\\
v_t=d_2v_{xx}+1-\alpha v+\beta u, & x\in(0,L),t>0,\\
u_x(x,t)=v_x(x,t)=0,&x=0,L,t>0,\\
u(x,0)=u_0(x),v(x,0)=v_0(x),&x \in (0,L),
\end{array}
\right.
\end{equation}
where $d_1$, $d_2$, $\mu, \lambda$ and $\alpha$, $\beta$ are positive constants.  For $d_1=\alpha=\beta=1$, $d_2=\mu=0$, and $\Phi=\ln v$, we see that (\ref{14}) the first equations are reduced to (\ref{13}).  Before formulating our main results, we want to mention that the density dependent sensitivity function $\Phi(v)$ plays an essential role in the global existence of classical solutions of (\ref{14}).  Two prototypical choices of the sensitivity functions are the linear case $\Phi=v$ and logarithmic case $\Phi(v)=\ln v$, and both are of biological importance and mathematical interest is the logarithmic function $\Phi(v)=\ln v$.  This form is chosen to model the stimulus perception governed by the celebrated Weber--Fecher's law.  Moreover, the singularity seems to be essential in the formation of travelling bands in \cite{KS}.  See the discussions in the last section of \cite{LW}.

For the linear case, one can prove by slightly modifying the arguments in \cite{OY} that, given any initial data $(u_0,v_0)\in H^1(0,L)\times H^1(0,L)$, (\ref{14}) has a unique bounded classical solution $(u,v)$.  The logarithmic case is much more complicated and difficult.  For $\lambda=\mu=0$, the global existence of the multi--dimensional parabolic--elliptic counterpart of (\ref{14}) has been studied in \cite{FWY}, \cite{STW}, \cite{Wk}, etc, provided that $\frac{\chi}{d_1}$ is not large.  Blow--up of radial solutions are investigated by Nagai and Senba in \cite{NS}.  For $\lambda,\mu>0$, it follows from Corollary 2.5 of \cite{STW} that the solution of (\ref{14}) is global in time, provided that $\frac{\chi}{d_1}<0$.  See the survey paper \cite{HP} of Hillen and Painter for works on chemotaxis models with other sensitivity functions.

It is an interesting and also important question to pursue if blow--up occurs when the smallness assumption on $\frac{\chi}{d_1}$ is relaxed.  The literature on global existence suggests that a singularity can be inhibited by the linear degradation of the cellular population, however, whether or not this is sufficient to exclude the blow--up for large $\frac{\chi}{d_1}$ remains open so far, in particular for high space dimensions.  Moreover, the singularity of the logarithmic sensitivity function at $v=0$ tends to support the formation of blow--up solutions.  See \cite{FWY}, \cite{Wk} and the references therein for detailed discussions in this aspect.

In contrast to the global existence question, it the main purpose of our paper to study the existence and stability of nonconstant positive stationary solutions of (\ref{11}).  In particular, we are concerned with the effect of the sensitivity function on the formation of nontrivial patterns.  To elucidate our goal, and also for the simplicity of our analysis, we introduce the parameters
\[\tilde d_1=\frac{d_1}{\mu},\tilde d_2=\frac{\mu+\beta\lambda}{\alpha \mu}d_2,\tilde \chi=\frac{\chi(\mu+\beta\lambda)}{\alpha\mu^2},\tilde \lambda=\frac{\lambda}{\mu},\]
and the transformation
\[(u,v)=\Big(\tilde u,\frac{\mu+\beta\lambda}{\alpha \mu} \tilde v\Big).\]
Then we see that the steady state of (\ref{11}) is reduced into
\begin{equation}\label{15}
\left\{
\begin{array}{ll}
(d_1u'-\chi u\Phi'(v)v')'+\lambda - u=0,&x \in (0,L),\\
d_2v''+1-(1+\lambda )v+u=0,& x\in(0,L),\\
u'(x)=v'(x)=0,&x=0,L,
\end{array}
\right.
\end{equation}
where we have dropped the tildes in (\ref{15}).  It is easy to see that $(\bar u,\bar v)=(\lambda,1)$ is the unique trivial solution of (\ref{15}).  We shall consider (\ref{11}) and its stationary system (\ref{15}) throughout the rest part of our paper.

This paper is organized as follows.  In section \ref{sec2}, we show that the homogeneous steady state $(\bar u,\bar v)$ is the global attractor of (\ref{11}) when $\chi=0$, independent of the initial data and the size of diffusions.  We also perform the linearized stability analysis of $(\bar u,\bar v)=(\lambda,1)$ when the chemoattraction rate $\chi>0$.  It is shown that large $\chi$ tends to destabilize $(\bar u,\bar v)$.  In section \ref{sec3}, we establish the existence of nonconstant positive steady states of (\ref{15})--see Theorem \ref{theorem31}.  The stability of the nonconstant solutions is also obtained for both linear and logarithmic sensitivity function--see Theorem \ref{theorem33} and Theorem \ref{theorem34}.  We present in section \ref{sec4} some numerical simulations to illustrate the emergence of nontrivial patterns.  Solutions with striking structures such as interior spikes, boundary layers, etc. are also presented.  Finally, we include some concluding remarks and discussions in section \ref{sec5}.

\section{Preliminary results and advection--driven instability}\label{sec2}

First of all, we show that the existence of nonconstant positive steady states of system (\ref{11}) is induced by the presence of chemotactic effect.  To this end, we study the dynamics of the system without the chemotactic term, i.e., the following system with $\chi=0$,
\begin{equation}\label{21}
\left\{
\begin{array}{ll}
u_t=d_1u_{xx}+\lambda- u,& x \in (0,L),t>0,\\
v_t=d_2v_{xx}+1-(1+\lambda) v+u, & x\in(0,L),t>0,\\
u_x(x,t)=v_x(x,t)=0,&x=0,L,t>0,\\
u(x,0)=u_0(x) \geq 0,~v(x,0)=v_0(x) \geq 0, &x\in (0,L).\\
\end{array}
\right.
\end{equation}
We have the following global stability result of the trivial steady state $(\bar u,\bar v)=(\lambda,1)$ and the proof of this result is same as that of Proposition 2.2 in \cite{MOW}.
\begin{proposition}\label{proposition1}
The positive equilibrium $(\bar u,\bar v)$ of (\ref{21}) is asymptotically stable and system (\ref{21}) does not have any nonconstant steady state.
\end{proposition}
\begin{proof}
First of all, the corresponding linearized system of (\ref{21}) at $(\bar u,\bar v)$ has the Jacobian matrix of the following form
\begin{equation*}
\mathcal{M}=\begin{pmatrix}
-1 & 0 \\
1 &-(1+\lambda)
\end{pmatrix},
\end{equation*}
which has a positive determinant $\vert \mathcal{M}\vert=1+\lambda>0$. Then $(\bar u,\bar v)$ is the only steady state of (\ref{21}) according to Theorem 3.1 of \cite{LN}.

On the other hand, it follows from straightforward calculations that the linearized matrix corresponding to system (\ref{21}) at $(\bar u,\bar v)$ is
\begin{equation*}
\begin{pmatrix}
-d_1\big(\frac{k\pi}{L}\big)^2-1  & 0 \\
1& -d_2\big(\frac{k\pi}{L}\big)^2-1-\lambda
\end{pmatrix},
\end{equation*}
and it has two negative eigenvalues
\[\eta_1=-d_1\Big(\frac{k\pi}{L}\Big)^2-1<0,~\eta_2=-d_2\Big(\frac{k\pi}{L}\Big)^2-1-\lambda<0,\]
therefore, $(\bar u,\bar v)$ is locally stable.   Moreover, similar to the analysis in \cite{MOW} or \cite{S}, one can show that system (\ref{21}) generates a strongly monotone semi--flow on $C([0,L],\mathbb{R}^2)$ with respect to $\{(u,v)\in C([0,L],\mathbb{R}^2) \vert u>0,v>0 \}$, hence $(\bar u,\bar v)$ is globally asymptotically stable.
\end{proof}

It follows from Proposition \ref{proposition1} that the diffusions alone does not change the dynamics of (\ref{21}) and no Turing's pattern emerges from this diffusive system.  Unlike diffusions, chemotaxis has the effect of destabilizing constant homogeneous solutions.  To elucidate this effect, we study the stability of $(\bar u,\bar v)$ with respect to (\ref{11}).  Let $(u,v)$ be any positive solution of (\ref{11}) and we put $(u,v)=(\bar u,\bar v)+(U,V)$, where $U$ and $V$ are spatially heterogeneous perturbations away from $(\bar u,\bar v)$, then one can easily have that
\begin{equation*}
\left\{
\begin{array}{ll}
U_t\approx d_1 U_{xx}- \chi \bar u \Phi'(\bar v)V_{xx}- U, &x\in (0,L),t>0,\\
V_t\approx d_2 V_{xx}+U-(1+\lambda)V ,&x \in(0,L),t>0,\\
U_x(x,t)=V_x(x,t)=0,&x=0,L,t>0,
\end{array}
\right.
\end{equation*}
By the standard linearized stability analysis--see Theorem 8.6 in \cite{Si} for example, the stability of $(\bar u,\bar v)$ with respect to (\ref{11}) can be determined by the eigenvalues of the linearized stability matrix.  To be precise, we have the following result.
\begin{proposition}\label{proposition2}
The constant solution $(\bar u,\bar v)$ of (\ref{15}) is unstable if and only if
\begin{equation}\label{22}
\chi>\chi_0=\min_{k\in \mathbb{N^{+}}} \frac{(d_1(\frac{k\pi}{L})^2+1)(d_2(\frac{k\pi}{L})^2+1+\lambda)}{ \bar u\Phi'( \bar v)(\frac{k\pi}{L})^2}.
\end{equation}
\end{proposition}
\begin{proof}
The $k$--th eigenvalue of $-\frac{d^2}{dx^2}$ under homogeneous Neumann boundary condition is $\lambda_k=\big(\frac{k\pi}{L}\big)^2$ and the stability matrix of $(\bar u,\bar v)$ with respect to (\ref{11}) is
\begin{equation}\label{23}
J_k=
\begin{pmatrix}
-d_1\big(\frac{k\pi}{L}\big)^2-1 & \chi  \bar u \Phi'( \bar v)\big(\frac{k\pi}{L}\big)^2 \\
1 &-d_2 \big(\frac{k\pi}{L}\big)^2 -1-\lambda
\end{pmatrix}.
\end{equation}
The characteristic polynomial of (\ref{23}) takes the form
\[\mathcal{P}(r)=r^2+\text{Tra}r+\text{Det},\]
where
\[\text{Tra}=\Big(d_1+d_2\Big)\Big(\frac{k\pi}{L}\Big)^2+2+\lambda,\]
\[\text{Det}=\Big(d_1\big(\frac{k\pi}{L}\big)^2+1\Big)\Big(d_2(\frac{k\pi}{L}\big)^2+1+\lambda\Big)-\chi\bar u\Phi'(\bar v)
\big(\frac{k\pi}{L}\big)^2.\]
Then we can see that $\mathcal{P} (r)$ has one positive root if and only if $\text{Det}<0$ and (\ref{22}) readily follows through simple calculations.  This completely finishes the proof of this proposition.
\end{proof}
One quick implication of the above linearized stability analysis is that, chemo--attraction ($\chi>0$) has the effect of destabilizing the constant steady state, while chemo--repulsion ($\chi<0$) tends to support the constant steady state.  Moreover, as we shall see in the coming analysis, nonconstant positive solutions of (\ref{15}) emerges from $(\bar u,\bar v)$ as $\chi$ surpasses $\chi_0$.  It is clear that the existence of such nontrivial solutions is due to the presence of large chemotactic coefficient (or advection rate) $\chi$ and we refer this as advection--driven instability of ($\bar u,\bar v$).

\section{Existence and stability of nonconstant positive steady states}\label{sec3}

In this section, we study the stationary system (\ref{15}) and investigate its nonconstant positive solutions.  For this purpose, we shall apply the Crandall--Rabinowitz bifurcation theories \cite{CR} on (\ref{15}).  To begin with, we write (\ref{15}) into the following abstract form
\[\mathcal{F}(u,v,\chi)=0, (u,v,\chi)\in \mathcal{X} \times \mathcal{X} \times \mathbb{R}^{+},\]
where the operator is given by
\begin{equation}\label{31}
\mathcal{F}(u,v,\chi)=
\begin{pmatrix}
(d_1u'-\chi u\Phi'(v)v')'+\lambda - u \\
d_2v''+1-(1+\lambda)v+u
\end{pmatrix},
\end{equation}
and the Hilbert spaces are
\[\mathcal{X}=H_N^2(0,L)=\{ w \in H^2(0,L) \vert w'(0)=w'(L)=0\},\]
and $\mathcal{Y}=L^2(0,L)$.  It is easy to see that $\mathcal{F}$ is a continuously differentiable mapping from $\mathcal{X}\times \mathcal{X}\times \mathbb{R}$ to $\mathcal{Y} \times \mathcal{Y}$ and $\mathcal{F}(\bar u,\bar v,\chi)=0$ for any $\chi \in \mathbb{R}$, and it follows from straightforward calculations, for any fixed $(\hat{u},\hat{v}) \in \mathcal{X} \times \mathcal{X}$, that the Fr$\acute{\text{e}}$chet derivative of $\mathcal{F}$ is given by
\begin{align}\label{32}
& D_{(u,v)}\mathcal{F}(\hat{u},\hat{v},\chi)(u,v) \nonumber\\
= &\begin{pmatrix}
d_1u''-\chi\big(u\Phi'(\hat{v})v'_1+\hat{u}\Phi'(\hat{v})v' + \hat{u}\Phi''(\hat{v})v'_1v\big)'- u \\
d_2v''-(1+\lambda)v+u
\end{pmatrix},
\end{align}
moreover, the derivative $D_{(u,v)}\mathcal{F}(\hat{u},\hat{v},\chi):\mathcal{X}\times \mathcal{X}\times \mathbb{R}\rightarrow \mathcal{Y} \times \mathcal{Y}$ is a Fredholm operator with zero index.  To see this, we denote $\textbf{u}=(u,v)^\text{T}$ and write (\ref{38}) into
\[D_{(u,v)}\mathcal{F}(\hat{u},\hat{v},\chi)(u,v)=\textbf{A}_0 \textbf{u}''+\textbf{F}_0(x,\textbf{u},  \textbf{u}'),\]
where
\[\textbf{A}_0=\begin{pmatrix}
d_1 & -\chi \hat{u} \Phi'(\hat{v}) \\
0 &  d_2
\end{pmatrix}\]
and
\[\textbf{F}_0=\begin{pmatrix}
-\chi(u\Phi'(\hat{v})v'_1+ \hat{u}\Phi''(\hat{v})v'_1v)'-\chi(\hat{u}\Phi'(\hat{v}))'v'\\
-(1+\lambda)v u
\end{pmatrix}.\]  It is obvious that operator (\ref{32}) is strictly elliptic since $d_1, d_2>0$; moreover it satisfies the Agmon's condition according to Remark 2.5 of case 2 with $N=1$ in Shi and Wang \cite{SW}.  Therefore, $D_{(u,v)}\mathcal{F}(\hat{u},\hat{v},\chi)$ is a Fredholm operator with zero index thanks to Theorem 3.3 and Remark 3.4 of \cite{SW}.

\subsection{Existence of nonconstant positive steady states}
We now seek the existence of nonconstant positive solutions of (\ref{15}) by applying the local bifurcation theory of Crandall--Rabinowtiz from \cite{CR}.  We shall take the attraction--rate $\chi$ as the bifurcation parameter.  First of all, we find the potential values.  If the bifurcation occurs at $(\bar u,\bar v,\chi)=(\lambda,1,\chi)$, we need the implicit function theorem to fail at this equilibrium, i.e., the following necessary condition is satisfied,
\[\mathcal{N}(D_{(u,v)}\mathcal{F}(\bar u ,\bar v,\chi) )\neq 0,\]
where $\mathcal{N}$ denotes the null space.  Taking $(\hat{u},\hat{v})=(\bar u,\bar v)$ in (\ref{32}), one has that
\begin{equation*}
D_{(u,v)}\mathcal{F}(\bar u,\bar v,\chi)(u,v)=
\begin{pmatrix}
d_1u''-\chi \bar u\Phi'(\bar v)v''- u \\
d_2v''-(1+\lambda)v+u
\end{pmatrix},
\end{equation*}
then the null space above consists of solutions to the following system,
\begin{equation}\label{33}
\left\{
\begin{array}{ll}
d_1u''-\chi \bar u\Phi'(\bar v)v'' - u=0, &x\in(0,L),\\
d_2v''-(1+\lambda)v+u=0, &x\in(0,L),\\
u'(x)=v'(x)=0, &x=0,L.
\end{array}
\right.
\end{equation}
To prove the necessary condition, we let $(u,v)$ be a solution of (\ref{33}) and expand it as
\[u(x)=\sum_{k=0}^{\infty} \bar u_k, v(x)=\sum_{k=0}^{\infty}  \bar v_k ,\]
where
\[ \bar u_k=T_k\cos\frac{k\pi x}{L} , \bar v_k=S_k\cos\frac{k\pi x}{L} ,\]
and $T_k$, $S_k$ are some constants to be determined.  Substituting the series into the (\ref{33}), we have that $(T_k,S_k)$ satisfies
\begin{equation}\label{34}
\begin{pmatrix}
-d_1(\frac{k\pi}{L})^2-1 & \chi  \bar u \Phi'( \bar v)(\frac{k\pi}{L})^2 \\
1 &-d_2(\frac{k\pi}{L})^2-1-\lambda
\end{pmatrix}
\begin{pmatrix}
T_k\\
S_k
\end{pmatrix}
=\begin{pmatrix}
0\\
0
\end{pmatrix},
\end{equation}
Since $(u(x),v(x))$ is nontrivial, we must have that (\ref{34}) has at least one nonzero solution, therefore the coefficient matrix must be singular and
\begin{equation*}
\begin{vmatrix}
-d_1\big(\frac{k\pi}{L}\big)^2-1 & \chi  \bar u \Phi'( \bar v)\big(\frac{k\pi}{L}\big)^2 \\
1 &-d_2\big(\frac{k\pi}{L}\big)^2-1-\lambda
\end{vmatrix}
=0.
\end{equation*}
It is easy to see that $k=0$ can be easily ruled out. For each $k\in\mathbb{N}^+$, we obtain from straightforward calculations the following bifurcation values $\chi$, which we shall denote by $\chi_k$ from now on
\begin{equation}\label{35}
 \chi_k=\frac{(d_1(\frac{k\pi}{L})^2+1)(d_2(\frac{k\pi}{L})^2+1+\lambda)}{ \bar u\Phi'( \bar v)(\frac{k\pi}{L})^2}, k\in\mathbb{N}^{+};
\end{equation}
moreover, dim $\mathcal{N}\big(D(u,v) \mathcal{F}( \bar u, \bar v,\chi_k)\big)=1$ and $\mathcal{N}(D_{(u,v)} \mathcal{F}( \bar u, \bar v,\chi_k))=\text{span}\{( \bar u_k, \bar v_k )\}$, where
\begin{equation}\label{36}
( \bar u_k, \bar v_k )=\Big(Q_k \cos\frac{k\pi x}{L}, \cos\frac{k\pi x}{L}\Big),
\end{equation}
and
\begin{equation}\label{37}
Q_k=d_2\Big(\frac{k\pi}{L}\Big)^2+1+\lambda, k\in \mathbb{N}^{+}.
\end{equation}
We now present the first main result of this paper in the following theorem which asserts that the local bifurcation does occur at $(\lambda,1,\chi_k)$.
\begin{theorem}\label{theorem31}
Let $d_1$, $d_2$, and $\lambda$ be positive constants and assume that $\Phi(v)$ is $C^2$--smooth with $\Phi(v)>0$ for $v>0$.  Suppose that
\begin{equation}\label{38}
d_1d_2 j^2 k^2 \Big(\frac{\pi}{L}\Big)^4\neq \lambda+1 \text{  for all positive integers k}\neq j.
\end{equation}
Then for any positive integer $k\in \mathbb{N}^{+}$, there exist a constant $\delta >0$ such that (\ref{15}) admits nonconstant positive bifurcating solutions $(u_k(s,x),v_k(s,x),\chi_k(s))$ for $s\in(-\delta,\delta)$, where $\chi_k(s)$ is a continues function of $s$ and $(u_k(s,x),v_k(s,x)) \in \mathcal{X} \times \mathcal{X}$; moreover, the bifurcation branch $\Gamma_k(s)$ around $( \bar u, \bar v,\chi_k)$ can be parameterized as
\[\chi_k(s)=\chi_k+O(s),(u_k(s,x),v_k(s,x))=(\bar u, \bar v)+s(Q_k,1)\cos\frac{k\pi x}{L}+s(\xi(s),\zeta(s)),\]
where $(\xi(s),\zeta(s))$ is an element in the closed complement $\mathcal{Z}$ of $\mathcal{N}(D_{(u,v)}\mathcal{F}(\bar u,\bar v,\chi_k))$ in $ \mathcal{X} \times \mathcal{X}$ with $(\xi_k(0),\zeta_k(0))=(0,0)$, where
\begin{equation}\label{39}
\mathcal{Z}=\big\{(u,v)\in \mathcal{X}\times \mathcal{X} \vert \int_0^L u\bar u_k+v\bar v_kdx=0\big\};
\end{equation}
furthermore, $(u_k(s,x),v_k(s,x),\chi_k(s))$ solves system (1.2) and all nonconstant solutions of (\ref{15}) around $( \bar u, \bar v,\chi_k)$ must stay on the curve $\Gamma_k(s)$.
\end{theorem}

\begin{proof}
We shall seek the existence of the bifurcating solutions by the Crandall--Rabinowitz local bifurcation theory in \cite{CR}.  To this end, we have verified all but the following transversality condition
\[\frac{d}{d\chi}D_{(u,v)}\mathcal{F}( \bar u, \bar v,\chi)( \bar u_k, \bar v_k) \vert_{\chi=\chi_k} \notin \mathcal{R}(D_{(u,v)} \mathcal{F}( \bar u, \bar v,\chi_k)),\]
where $ \mathcal{R}$ denotes the range and
\begin{equation*}
\frac{d}{d\chi}D_{(u,v)}\mathcal{F}( \bar u, \bar v,\chi)( \bar u_k, \bar v_k) \vert_{\chi=\chi_k} =
\begin{pmatrix}
-\bar{u}\Phi'( \bar v)  \bar v''_k \\
0
\end{pmatrix}.
\end{equation*}
To show the transversality condition, we argue by contradiction as follows.  Suppose that there exist a nontrivial solutions $(\tilde u,\tilde v)$ for which the transversality condition fails, i.e., $(\tilde u,\tilde v)$ satisfies
\begin{equation*}
\begin{pmatrix}
d_1 \tilde{u}''-\chi \bar u\Phi'( \bar v)\tilde{v}''- \tilde{u} \\
d_2 \tilde{v}''+\tilde{u}-(1+\lambda)\tilde{v}
\end{pmatrix}
=
\begin{pmatrix}
- \bar u \Phi'( \bar v)  \bar v''_k \\
0
\end{pmatrix}.
\end{equation*}
Substituting the following eigen--expansion of $\tilde{u}$ and $ \tilde{v}$
\[\tilde{u}=\sum_{k=0}^{\infty} \tilde{T}_k \cos\frac{k \pi x}{L} , \tilde{v}=\sum_{k=0}^{\infty} \tilde{S}_k \cos\frac{k \pi x}{L},\]
into system, we arrive at the following system
\begin{equation}\label{310}
\begin{pmatrix}
-d_1\big(\frac{k\pi}{L}\big)^2-1 & \chi  \bar u \Phi'( \bar v)\big(\frac{k\pi}{L}\big)^2 \\
1 &-d_2\big(\frac{k\pi}{L}\big)^2-1-\lambda
\end{pmatrix}
\begin{pmatrix}
\tilde{T}_k \\
\tilde{S}_k
\end{pmatrix}
=
\begin{pmatrix}
 \bar u \Phi'( \bar v) \big(\frac{k\pi}{L}\big)^2 S_k \\
0
\end{pmatrix}.
\end{equation}
Now we see that the coefficient matrix is singular in light of (\ref{37}), while the right hand side is nonzero.  This leads to a contradiction.  Therefore the transversality condition is verified and the statements of Theorem \ref{theorem31} follows from the Crandall--Rabinowitz bifurcation theory in \cite{CR}.

Finally, since $\chi_k \neq \chi_j $ for all positive integers $k \neq j$, i.e.,
\[\frac{(d_1\big(\frac{k\pi}{L}\big)^2+1)\big(d_2(\frac{k\pi}{L}\big)^2+1+\lambda)}{ \bar u\Phi'( \bar v)\big(\frac{k\pi}{L}\big)^2} \neq \frac{(d_1\big(\frac{j\pi}{L}\big)^2+1)(d_2\big(\frac{j\pi}{L}\big)^2+1+\lambda)}{ \bar u\Phi'( \bar v)\big(\frac{j\pi}{L}\big)^2},\]
then (\ref{38}) follows from simple calculations.
\end{proof}
We see from Proposition \ref{proposition2} and Theorem \ref{theorem31} that the bifurcation occurs at the exact location where the homogeneous steady state $(\bar u,\bar v)$ changes its stability, then nonconstant steady states emerges through bifurcations.  Apparently, this is due to the presence of the chemotaxis $\chi$.

\subsection{Stability analysis of the solution bifurcation from $(\bar u,\bar v,\chi_k)$.}
We shall now investigate the stability of the spatially inhomogeneous steady state $(u_k(s,x),v_k(s,x))$ established in Theorem \ref{theorem31}.  The stability or instability here refers to that of the inhomogeneous solution taken as an equilibrium to (\ref{11}).  To this end, we apply the classical results from Crandall and Rabinowitz \cite{CR2} on the linearized stability with an analysis of the spectrum of system (\ref{15}).

First of all, we shall show that the bifurcation curve $\Gamma_k(s)$ is of pitch--fork for each $k\in\mathbb{N}^+$.  Then we determine the direction in which $\Gamma_k(s)$ turns around $(\bar u,\bar v,\chi_k)$.  Let $\Phi(v)$ be $C^4$--smooth, then we have that $\mathcal{F}$ is $C^4$--smooth and $(u_k,v_k,\chi_k)$ are $C^3$--smooth functions of $s$, and we have the following asymptotic expansions,
\begin{equation}\label{311}
\left\{
\begin{array}{ll}
u_k(s,x)= \bar u+sQ_k\cos\frac{k\pi x}{L}+s^2\psi_1 (x)+s^3\psi_2 (x)+o(s^3), \\
v_k(s,x)= \bar v+s\cos\frac{k\pi x}{L}+s^2\varphi_1 (x)+s^3\varphi_2 (x)+o(s^3), \\
\chi_k(s)=\chi_k+\mathcal{K}_2s+\mathcal{K}_3s^2+o(s^2),
\end{array}
\right.
\end{equation}
where $Q_k$ is defined in (\ref{37}), $(\psi_i,\varphi_i) \in \mathcal{Z} $ for $i=1,2 $, and $\mathcal{K}_2$, $\mathcal{K}_3$ are some constants to be determined.  We want to mention that the $o(s^3)$ terms in $u_k(s,x)$ and $v_k(s,x)$ are taken in $H^2$--norms.  To evaluate $\mathcal{K}_2$ and $\mathcal{K}_3$, we need the following identities, which can be easily obtained from straightforward calculations,
\begin{align}\label{312}
&d_1u_k''(s,x)=-d_1\Big(\frac{k\pi}{L}\Big)^2 Q_k \cos\frac{k\pi x}{L}+d_1\psi''_1  s^2+d_1\psi''_2s^3+o(s^3),\\
&v'_k(s,x)=-\frac{k\pi }{L}\sin \frac{k\pi x}{L}s+\varphi'_1s^2+\varphi'_2s^3+o(s^3), \label{313}
\end{align}
and
\begin{align}\label{314}
& u_k(s,x)\Phi'(v_k(s,x))=\bar u\Phi'( \bar v)+\Big(\Phi'( \bar v)Q_k\cos\frac{k\pi x}{L}+ \bar u\Phi''( \bar v)\cos\frac{k\pi x}{L}\Big)s\\
&+\Big(\Phi'(\bar v)\psi_1+\bar u\Phi''(\bar v)\varphi_1+\big(\frac{1}{2}\bar u\Phi'''(\bar v)+\Phi''(\bar v)Q_k\big)\cos^2\frac{k\pi x}{L} \Big) s^2+o(s^3)\nonumber,
\end{align}
Substituting (\ref{311})-(\ref{314}) into the $u$--equation of (\ref{15}), we have from simple calculations that
\begin{align}\label{315}
& s{d_1\Big(\frac{k\pi}{L}\Big)^2Q_k\cos\frac{k\pi x}{L}-s^2d_1\psi''_1 -s^3d_1\psi''_2} \nonumber\\
=& -s\Big( Q_k\cos\frac{k\pi x}{L}-\chi_k P_0 (\frac{k\pi }{L})^2\cos\frac{k\pi x}{L} \Big)-     s^2\Big(\psi_1+\chi_k P_0\varphi''_1-\chi_k\frac{k\pi }{L}\big(P_1 \sin\frac{k\pi x}{L}\big)'\nonumber\\
&-\mathcal{K}_2 P_0( \frac{k\pi }{L})^2\cos\frac{k\pi x}{L} \Big)
-s^3\Big(  \psi_2+\chi_k P_0\varphi''_2+\chi_k(P_1\varphi'_1)'+\mathcal{K}_2 P_0\varphi''_1\nonumber\\
&-\mathcal{K}_2\frac{k\pi}{L}\big(P_1 \sin\frac{k\pi x}{L}\big)'-\mathcal{K}_3P_0( \frac{k\pi }{L} )^2 \cos\frac{k\pi x}{L} -\chi_k \frac{k\pi }{L}\big(P_2\sin\frac{k\pi x}{L} \big)'\Big).
\end{align}
where we have used the notations
\[P_0= \bar u\Phi'(\bar v),\text{  }P_1=\Phi'( \bar v)Q_k\cos\frac{k\pi x}{L}+ \bar u \Phi''(\bar v)\cos\frac{k\pi x}{L},\]and
\[P_2=\Phi'( \bar v)\psi_1+ \bar u \Phi''( \bar v)\varphi_1+\Big(\frac{1}{2} \bar u\Phi'''(\bar v)+\Phi''(\bar v)Q_k\Big)\cos^2\frac{k\pi x}{L}.\]

Equating the $s^2$--terms in (\ref{315}), we have that
\begin{align}\label{316}
& d_1\psi''_1 - \psi_1+(\frac{k\pi}{L})^2\mathcal{K}_2 P_0\cos\frac{k\pi x}{L}  \nonumber\\
=& \chi_k\Big( \bar u\Phi'( \bar v)\varphi''_1-\big(\frac{k\pi}{L}\big)^2(\Phi'( \bar v)Q_k+ \bar u\Phi''
 ( \bar v))\cos\frac{2k\pi x}{L}\Big).
\end{align}
Multiplying (\ref{316}) by $\cos\frac{k\pi x}{L}$ and then integrating it over $(0,L)$ by parts, we have that
\begin{align}\label{317}
& \Big(\frac{k\pi}{L}\Big)^2 \bar u\Phi'( \bar v)\mathcal{K}_2\int_0^L \cos^2 \frac{k\pi x}{L} dx  \nonumber\\
=&\Big(d_1\big(\frac{k\pi}{L}\big)^2+1\Big)\int_0^L \psi_1 \cos\frac{k\pi x}{L} dx-\Big(\frac{k\pi}{L}\Big)^2\chi_k \bar u \Phi'(\bar v )\int_0^L \varphi_1 \cos\frac{k\pi x}{L} dx.
\end{align}
Similarly, we can substitute (\ref{311})--(\ref{314}) into the $v$--equation of (\ref{35}) and collect that
\begin{align}\label{318}
& -sd_2\Big(\frac{k\pi}{L}\Big)^2\cos\frac{k\pi x}{L}+s^2d_2\varphi''_1+s^3d_2\varphi''_2+o(s^3) \nonumber\\
 =&s \Big((1+\lambda)\cos\frac{k\pi x}{L}-Q_k\cos\frac{k\pi x}{L}\Big)+s^2\big((1+\lambda)\varphi_1-\psi_1\big)\nonumber\\
 &+s^3\big((1+\lambda)\varphi_2-\psi_2\big)+o(s^3).
\end{align}
Equating the $s^2$--terms in (\ref{318}), we get that
\begin{equation}\label{319}
d_2\varphi''_1-(1+\lambda)\varphi_1+\psi_1=0.
\end{equation}
Multiplying (\ref{319}) by $\cos\frac{k\pi x}{L}$ and integrating it over $(0,L)$ by parts leads us to
\begin{equation}\label{320}
\Big(d_2\big(\frac{k\pi}{L}\big)^2+1+\lambda \Big)\int_0^L \varphi_1 \cos\frac{k\pi x}{L} dx- \int_0^L \psi_1 \cos\frac{k\pi x}{L} dx=0.
\end{equation}
On the other hand, since $(\psi_1,\varphi_1)\in \mathcal{Z}$, we have that
\begin{equation}\label{321}
Q_k\int_0^L \psi_1\cos\frac{k\pi x}{L}dx + \int_0^L \varphi_1\cos\frac{k\pi x}{L}dx=0.
\end{equation}
From (\ref{320}) and (\ref{321}), we arrive at the following system
\begin{equation}\label{322}
\begin{pmatrix}
-1 & d_2\big(\frac{k\pi}{L}\big)^2+1+\lambda \\
Q_k & 1
\end{pmatrix}
\begin{pmatrix}
\int_0^L \psi_1 \cos\frac{k\pi x}{L} dx \\
\int_0^L \varphi_1 \cos\frac{k\pi x}{L} dx
\end{pmatrix}
=
\begin{pmatrix}
0 \\
0
\end{pmatrix}.
\end{equation}
Since the coefficient matrix is not singular, we can easily have that
\begin{equation}\label{323}
\int_0^L \psi_1 \cos\frac{k\pi x}{L} dx =\int_0^L \varphi_1 \cos\frac{k\pi x}{L} dx =0,
\end{equation}
therefore $\mathcal{K}_2=0$ follows from (\ref{317}) and (\ref{323}).  Hence the bifurcation branch $\Gamma_k(s)$ around $(\bar u,\bar v,\chi_k)$ is of pitch--fork type, i.e., is of one--sided.

We continue with the stability analysis of the bifurcation solutions $(u_k(s,x),v_k(s,x))$ with $s\in (-\delta,\delta)$.  We linearize system (\ref{11}) at the bifurcation solution, then by the principle of exchange of stability--Theorem 1.16 in \cite{CR2}, the branch $\Gamma_k(s)$ will be asymptotically stable if the real part of any eigenvalue $\eta$ of the following problem is negative:
\begin{equation}\label{324}
D_{(u,v)}\mathcal{F}(u_k(s,x),v_k(s,x),\chi_k(s))(u,v)=\eta(u,v), (u,v)\in \mathcal{X} \times \mathcal{X}.
\end{equation}
Following the similar analysis in \cite{CR2}, or the proof of Theorem 5.5 in \cite{CKWW}, we have the following stability results.

\begin{proposition}\label{proposition3}
Suppose all conditions in Theorem \ref{theorem31} are satisfied.  For $s\in (-\delta,\delta)$, $s\neq0$, the bifurcating solution $(u_k(s,x),v_k(s,x))$ is asymptotically stable if $\mathcal{K}_3>0$ and unstable if $\mathcal{K}_3<0$.
\end{proposition}

\begin{proof}
We already know that $D_{(u,v)}\mathcal{F}(\bar u,\bar v,\chi_k):\mathcal{X}\times \mathcal{X}\times \mathbb{R} \rightarrow \mathcal{Y}\times \mathcal{Y}$ is a Fredholm operator with zero index and its null--space is of one--dimensional.  Moreover, similar to the analysis that leads to the transversality condition in Theorem \ref{theorem31}, we can show that $(u_k(s,x),v_k(s,x))\not \in \mathcal{R}(D_{(u,v)}\mathcal{F}(\bar u,\bar v,\chi_k))$.  Therefore, $\eta=0$ is a simple eigenvalue of $D_{(u,v)}\mathcal{F}(\bar{u},\bar{v},\chi_k)$ with eigenspace $\mathcal{N}\big(D_{(u,v)}\mathcal{F}(\bar{u},\bar{v},\chi_k)\big)=\{(Q_k,1)\cos \frac{k\pi x}{L}\}$.

Thanks to Corollary 1.13 in \cite{CR2}, there exist an interval $I$ and continuously differentiable functions $(\chi,s):I\times (-\delta,\delta) \rightarrow (\gamma(\chi),\eta(s))$ with $\chi_k\in I$ and $\gamma(\chi_k)=\eta(0)=0$ such that, $\gamma(\chi)$ is an eigenvalue of the following eigenvalue problem
\begin{equation}\label{325}
D_{(u,v)}\mathcal{F}(\bar{u},\bar{v},\chi)(u,v)=\gamma(u,v),~(u,v)\in \mathcal{X} \times \mathcal{X},
\end{equation}
and $\eta(s)$ is an eigenvalue of (\ref{324}).  Moreover, $\eta(s)$ is the only eigenvalue of (\ref{324}) in any fixed neighbourhood of the origin of the complex plane and the same assertion can be made about $\gamma(\chi)$ around $\gamma(\chi_k)$.  We also know from \cite{CR2} that the eigenfunction of (\ref{325}) are continuously differentiable and can be represented by $\big(u(\chi,x),v(\chi,x)\big)$, which is uniquely determined by $\big(u(\chi_k,x),v(\chi_k,x)\big)=\big( Q_k \cos \frac{k\pi x}{L},\cos \frac{k\pi x}{L} \big)$ and $\big(u(\chi,x),v(\chi,x)\big)-\big( Q_k\cos \frac{k\pi x}{L},\cos \frac{k\pi x}{L} \big) \in \mathcal{Z}$, where $Q_k$ and $\mathcal{Z}$ are defined in (\ref{37}) and (\ref{39}) respectively.  Moreover, it follows from (\ref{39}) that (\ref{325}) is equivalent to
\begin{equation}\label{326}
\left\{
\begin{array}{ll}
d_1 u''-\chi \bar{u}\Phi'(\bar{v})v''-u=\gamma u,& x\in(0,L),  \\
d_2 v''-(1+\lambda)v+u=\gamma v,& x\in(0,L),  \\
u'(x)=v'(x)=0,&x=0,L.
\end{array}
\right.
\end{equation}
Differentiating (\ref{326}) with respect to $\chi$ and then taking $\chi=\chi_k$, we have that
\begin{equation}\label{327}
\left\{\!\!\!
\begin{array}{ll}
d_1\dot{u}''\!\!-\!\bar{u}\Phi'(\bar{v})\big(\cos\frac{k\pi x}{L}\big)''\!\!-\chi_k\bar{u}\Phi'(\bar{v})\dot{v}''\!\!-\!\dot{u}=\dot{\gamma}(\chi_k) Q_k \cos\frac{k\pi x}{L}, &x\in(0,L), \\
d_2 \dot{v}''\!\!-(1+\lambda)\dot{v}+\dot{u}=\dot{\gamma}(\chi_k) \cos\frac{k\pi x}{L},&x\in(0,L),\\
\dot{u}'(x)=\dot{v}'(x)=0,&x=0,L,
\end{array}
\right.
\end{equation}
where we use the notations in (\ref{327}) $\dot{u}=\frac{\partial u(\chi,x)}{\partial \chi}\big\vert_{\chi=\chi_k}$, $\dot{v}=\frac{\partial v(\chi,x)}{\partial \chi}\big\vert_{\chi=\chi_k}$ and $''$ denotes the derivative taken with respect to $x$.

Multiplying both equations in (\ref{327}) by $\cos\frac{k \pi x}{L}$ and integrating them over $(0,L)$ by parts, we arrive at the following system
\[\begin{pmatrix}\!\!
-d_1 \!\big(\frac{k \pi}{L}\big)^2\!\!-\!1 & \chi_k \!\bar{u}\Phi'(\bar{v})\big(\frac{k \pi}{L}\big)^2   \\
~~\\
1 & -d_2\big(\frac{k \pi}{L}\big)^2\!\!-\!1-\!\lambda
\!\!\end{pmatrix}\!
\!\!\begin{pmatrix}\!
\int_0^L \!\!\dot{u}\cos\frac{k\pi x}{L}dx \\
~~\\
\int_0^L\!\! \dot{v}\cos\frac{k\pi x}{L}dx
\end{pmatrix}\!\!=\!\!\begin{pmatrix}
\!\Big(\dot{\gamma}(\chi_k)Q_k\!-\!\bar{u}\Phi'(\bar{v})\big(\frac{k\pi}{L} \big)^2\Big)\! \frac{L}{2}\\
~~\\
\dot{\gamma}(\chi_k)\frac{L}{2}
\!\!\!\!\end{pmatrix}.
\]
The coefficient matrix is singular due to (\ref{35}), therefore we must have that
\[ \frac{\dot{\gamma}(\chi_k)Q_k-\bar{u}\Phi'(\bar{v}) (\frac{k\pi}{L})^2}{\dot{\gamma}(\chi_k)}=
-\Big(d_1 \big(\frac{k \pi}{L}\big)^2+1\Big)\]
and this implies that
\[\dot{\gamma}(\chi_k)=\frac{\bar{u}\Phi'(\bar{v})\big(\frac{k\pi}{L} \big)^2}{d_1\big(\frac{k\pi}{L} \big)^2+Q_k}>0.\]
By Theorem 1.16 in \cite{CR2}, the functions $\eta(s)$ and $-s\chi'_k(s)\dot{\gamma}(\chi_k)$ have the same zeros and the same signs near $s=0$; moreover, for $\eta(s) \neq0$, one has that
\[\lim_{s\rightarrow 0}\frac{-s\chi'_k(s)\dot{\gamma}(\chi_k)}{\eta(s)}=1,\]
therefore we conclude that $\text{sgn} (-\eta(s))=\text{sgn}(\mathcal{K}_3)$ since $\mathcal{K}_2=0$.

If $\mathcal{K}_3<0$, the instability statements quickly follow from the positive sign of $\eta(s)$.  To show the stability part, we observe, for $s\in(-\delta,\delta)$, that (\ref{325}) has no nonzero eigenvalues with non--positive real parts if $\mathcal{K}_2>0$.  Then it follows from the standard perturbation theory that all eigenvalues of (\ref{324}) have no positive real part in a small neighborhood of the origin of the complex plane.  This completes the proof of Proposition \ref{proposition3}.
\end{proof}

We now proceed to find the sign of $\mathcal{K}_3$ to determine the turning direction and the stability of $\Gamma_k(s)$ around $(\bar u, \bar v,\chi_k)$.  For this purpose, we collect the $s^3$--terms of (\ref{315}) and (\ref{318}) and obtain the following system in light of $\mathcal{K}_2=0$,
\begin{equation}\label{328}
\left\{
\begin{array}{ll}
d_1\psi''_2-\psi_2-\chi_k\big(P_0\varphi'_2+P_1\varphi'_1-P_2(\frac{k\pi}{L})\sin\frac{k\pi x}{L})'+P_0\mathcal{K}_3\big(\frac{k\pi}{L}\big)^2\cos\frac{k\pi x}{L}=0,& x\in(0,L),\\
d_2\varphi''_2-(1+\lambda)\varphi_2+\psi_2=0,& x\in(0,L),\\
\psi'_2(x)=\varphi'_2=0 ,&x=0,L.
\end{array}
\right.
\end{equation}
where we have employed the notations $P_0$, $P_1$ and $P_2$ as in (\ref{315}).
Following the same calculations that lead to (\ref{323}), we can show that
\begin{equation}\label{329}
\int_0^L \psi_2\cos\frac{k\pi x}{L} dx=\int_0^L \varphi_2\cos\frac{k\pi x}{L} dx=0.
\end{equation}

Now multiplying the first equation of (\ref{328}) by $\cos\frac{k\pi x}{L}$ and then integrating it over $(0,L)$, we have that
\begin{align}\label{330}
& d_1\int_0^L \psi''_2\cos\frac{k\pi x}{L} dx-\int_0^L \psi_2\cos\frac{k\pi x}{L} dx-\chi_k \bar u \Phi'(\bar v)\int_0^L \varphi''_2\cos\frac{k\pi x}{L} dx \nonumber\\
&+\chi_k\big(\Phi'(\bar v) Q_k+ \bar u \Phi''(\bar v)\big)\frac{k\pi}{L}\int_0^L \varphi'_1\sin\frac{k\pi x}{L} \cos\frac{k\pi x}{L} dx-\chi_k\big(\Phi'(\bar v) Q_k+\bar u \Phi''(\bar v)\big)\nonumber\\
&\cdot\int_0^L \varphi''_1 \cos^2\frac{k\pi x}{L} dx +\frac{k\pi}{L}\chi_k\Phi'(\bar v)\int_0^L \psi''_1 \sin\frac{k\pi x}{L} \cos\frac{k\pi x}{L} dx+\frac{k\pi}{L}\chi_k\bar u \Phi''(\bar v)\nonumber\\
&\cdot \int_0^L \varphi'_1\sin\frac{k\pi x}{L} \cos\frac{k\pi x}{L} dx-\chi_k\big(\bar u \Phi'''(\bar v)+2\Phi''(\bar v)Q_k\big)\Big(\frac{k\pi}{L}\Big)^2\int_0^L \sin^2\frac{k\pi x}{L}\cos^2\frac{k\pi x}{L} dx\nonumber\\
&+\chi_k\Big(\frac{k\pi}{L}\Big)^2\Phi'(\bar v) \int_0^L \psi_1 \cos^2\frac{k\pi x}{L} dx
 +\chi_k\Big(\frac{k\pi}{L}\Big)^2\bar u \Phi''(\bar v) \int_0^L \varphi_1 \cos^2\frac{k\pi x}{L}dx \nonumber\\   &+\chi_k\Big(\frac{k\pi}{L}\Big)^2\Big(\frac{1}{2}\bar u \Phi'''+\Phi''(\bar v)Q_k\Big)\int_0^L \cos^4 \frac{k\pi x}{L} dx
 +\bar u \Phi'(\bar v)\Big(\frac{k\pi}{L}\Big)^2\mathcal{K}_3\int_0^L \cos^2\frac{k\pi x}{L} dx\nonumber\\
 &=0.
\end{align}
Putting (\ref{329}) and (\ref{330}) together, we obtain that
\begin{align}\label{331}
\frac{\bar u \Phi'(\bar v)}{2\chi_k}\mathcal{K}_3=&-\frac{\Phi'(\bar v) Q_k+\frac{1}{2}\bar u \Phi''(\bar v)}{L}\int_0^L \varphi_1\cos\frac{2k\pi x}{L}dx\nonumber\\
&+\frac{\Phi'(\bar v)}{2L}\int_0^L \psi_1 \cos\frac{2k\pi x}{L} dx-\frac{\bar u \Phi''(\bar v)}{2L}\int_0^L \varphi_1dx\\
&-\frac{\Phi'(\bar v)}{2L}\int_0^L \psi_1dx
-\frac{1}{8}\Big(\frac{1}{2}\bar u \Phi'''(\bar v)+\Phi''(\bar v)Q_k\Big). \nonumber
\end{align}
To find $\mathcal{K}_3$, we now proceed to evaluate the following integrals $\int_0^L \varphi_1\cos\frac{2k\pi x}{L}dx$, $\int_0^L \psi_1\cos\frac{2k\pi x}{L}dx$, $\int_0^L \varphi_1 dx$ and $\int_0^L \psi_1 dx$.  To obtain the first two integrals, we multiply (\ref{316}) and (\ref{319}) by $\cos\frac{2k\pi x}{L}$ and integrate them over $(0,L)$, then it follows from simple calculations that
\begin{align}\label{332}
&-\Big(d_1(\frac{2k\pi }{L})^2+1\Big)\int_0^L \psi_1\cos\frac{2k\pi x}{L}dx+\Big(\frac{2k\pi}{L}\Big)^2\chi_k\bar u \Phi'(\bar v)\int_0^L \varphi_1\cos\frac{2k\pi x}{L}dx \nonumber\\
=&-\frac{k^2\pi^2}{2L}\chi_k(\Phi'(\bar v) Q_k+\bar u \Phi''(\bar v)).
\end{align}
and
\begin{equation}\label{333}
\int_0^L \psi_1\cos\frac{2k\pi x}{L}dx-\Big(d_2\big(\frac{2k\pi}{L}\big)^2+1+\lambda\Big)\int_0^L \varphi_1\cos\frac{2k\pi x}{L}dx=0.
\end{equation}
Putting $\chi_k$ in (\ref{35}) into (\ref{332}), we solve (\ref{332}) and (\ref{333}) to obtain that
\begin{equation}\label{334}
\int_0^L \psi_1\cos\frac{2k\pi x}{L}dx=\frac{\frac{Q_kL}{2\bar u \Phi'(\bar v)}\big(d_1(\frac{k\pi}{L})^2+1\big)\big(\Phi'(\bar v) Q_k+\bar u \Phi''(\bar v)\big)\big(d_2(\frac{2k\pi}{L})^2+1+\lambda\big)}{12d_1d_2(\frac{k\pi}{L})^4-3(1+\lambda)},
\end{equation}
and
\begin{equation}\label{335}
\int_0^L \varphi_1\cos\frac{2k\pi x}{L}dx=\frac{\frac{Q_kL}{2\bar u \Phi'(\bar v)}\big(d_1\big(\frac{k\pi}{L}\big)^2+1\big)\big(\Phi'(\bar v)Q_k+\bar u \Phi''(\bar v)\big)}{12d_1d_2(\frac{k\pi}{L})^4-3(1+\lambda)}.
\end{equation}
We can see that the denominator in (\ref{334}) and (\ref{335}) is nonzero by taking $j=2k$ in (\ref{38}).

Finally, by integrating (\ref{316}) and (\ref{319}) over $(0,L)$, we have from simple calculations that
\begin{equation}\label{336}
\int_0^L \psi_1 dx=\int_0^L \varphi_1 dx=0.
\end{equation}
In light of (\ref{334})--(\ref{336}), we conclude from (\ref{331}) that
\begin{align}\label{337}
\frac{\bar u \Phi'(\bar v)}{2\chi_k}\mathcal{K}_3=&\frac{\frac{Q_k}{2\bar u \Phi'(\bar v)}(d_1(\frac{k\pi}{L})^2+1)(\Phi'(\bar v) Q_k+\bar u \Phi''(\bar v))(-\frac{1}{2}\bar u \Phi''(\bar v)+\Phi'(\bar v) Q_k-\frac{3}{2}\Phi'(\bar v)(1+\lambda))}{12d_1d_2(\frac{k\pi}{L})^4-3(1+\lambda)}\nonumber\\
& -\frac{1}{8}\Big(\frac{1}{2}
\bar u \Phi'''(\bar v)+\Phi''(\bar v)Q_k\Big).
\end{align}

It appears that $\mathcal{K}_3$ in (\ref{337}) is extremely complicated and it is very difficult to determine its sign for general $\Phi(v)$.  To elucidate the effect of the sensitivity function on the stability of the bifurcating solutions, and for the simplicity of our calculations, we shall divide our discussions into two cases, where $\Phi(v)=v$ and $\Phi(v)=\ln v$, respectively.
\subsection{System with linear sensitivity function}
In this section, we study the stability of the bifurcating solutions established in Theorem \ref{theorem31} for $\Phi(v)=v$.  Substituting $\Phi(\bar v)=\bar v$, $\Phi'(\bar v)=1$ and $\Phi''(\bar v)=0$ into (\ref{337}), we can easily get that
\begin{equation}\label{338}
\frac{\bar{u}\Phi'(\bar v)}{2\chi_k}\mathcal{K}_3=\frac{\big(Q_k-\frac{3}{2}(1+\lambda)\big)
(d_1(\frac{k\pi}{L})^2+1\big)\frac{Q_k^2}{2\bar{u}}}{12d_1d_2(\frac{k\pi}{L})^4-3(1+\lambda)}.
\end{equation}
To evaluate the sign of $\mathcal{K}_3$, we first recognize that $(d_1(\frac{k\pi}{L})^2+1)\frac{Q_k^2}{2\bar{u}}>0$ in (\ref{338}); moreover, we can substitute $Q_k=d_2(\frac{k\pi}{L})^2+1+\lambda$ into (\ref{338}) and then collect that
\begin{equation}\label{339}
\mathop{\rm sgn}\mathcal{K}_3=\mathop{\rm sgn} \Big(\frac{1+\lambda-2d_2(\frac{k\pi}{L})^2}{1+\lambda-4d_1d_2(\frac{k\pi}{L})^4}\Big).
\end{equation}
We are ready to present the following results that characterize of the sign of $\mathcal{K}_3$ hence the stability of the bifurcating solutions.
\begin{theorem}\label{theorem33}
Suppose that the conditions in Theorem \ref{theorem31} hold.  The bifurcation curve $\Gamma_k(s)$ of around $(\bar{u},\bar{v},\chi_k)$ turns to the right if $\mathcal{K}_3>0 $ and to the left if $\mathcal{K}_3<0$.  Moreover, we have the following results on the sign of $\mathcal{K}_3$:

(i).  if $d_1= \frac{1}{2}(\frac{L}{k\pi})^2$, then $\mathcal{K}_3>0$ for all $d_2\in(0,\infty)$;

(ii).  if $d_1\in(0,\frac{1}{2}(\frac{L}{k\pi})^2)$, then $\mathcal{K}_3>0$ for $d_2\in(0,\frac{\lambda+1}{4d_1}(\frac{L}{k\pi})^4)\cup (\frac{\lambda+1}{2}(\frac{L}{k\pi})^2,\infty) $ and $\mathcal{K}_3<0$ for $d_2\in(\frac{\lambda+1}{4d_1}(\frac{L}{k\pi})^4,\frac{\lambda+1}{2}(\frac{L}{k\pi})^2)$;

(iii).  if $d_1\in(\frac{1}{2}(\frac{L}{k\pi})^2,\infty)$, then $\mathcal{K}_3>0$ for $d_2\in(0,\frac{\lambda+1}{2}(\frac{L}{k\pi})^2)\cup (\frac{\lambda+1}{4d_1}(\frac{L}{k\pi})^4,\infty) $ and $\mathcal{K}_3<0$ for $d_2\in(\frac{\lambda+1}{2}(\frac{L}{k\pi})^2, \frac{\lambda+1}{4d_1}(\frac{L}{k\pi})^4)$.
\end{theorem}
\begin{proof}
If $d_1= \frac{1}{2}(\frac{L}{k\pi})^2$, we readily have that $\mathcal{K}_3=1$, which implies part (\emph{i}).

If $d_1 \neq\frac{1}{2}(\frac{L}{k\pi})^2$, we have from (\ref{339}) that
\[\mathop{\rm sgn}\mathcal{K}_3=\mathop{\rm sgn}\Big(d_2-\frac{\lambda+1}{2}\big(\frac{L}{k\pi}\big)^2\Big)\Big(d_2-\frac{\lambda+1}{4d_1}\big(\frac{L}{k\pi}\big)^4 \Big),\]
then (\emph{ii}) and (\emph{iii}) follows from simple calculations.
\end{proof}
We see from Theorem \ref{theorem33} that both large and small values of $d_2$ lead to the stability of the bifurcating solutions $(u_k(s,x),v_k(s,x))$, however the intermediate value of $d_2$ tends to destabilize the bifurcating solutions.

\subsection{System with logarithmic sensitivity function}
Now we evaluate the sign of $\mathcal{K}_3$ for $\Phi(v)=\ln v$.  Since $\bar v=1$, we have that $\Phi'(\bar v)=1, \Phi''(\bar v)=-1$ and $\Phi'''(\bar v)=2$, then (\ref{337}) gives us that
\begin{equation}\label{340}
\mathop{\rm sgn} \mathcal{K}_3=\frac{(Q_k-\lambda)(AQ_k^2+BQ_k+C)}{12d_1\big(\frac{k\pi}{L}\big)^2 Q_k-(12d_1\big(\frac{k\pi}{L}\big)^2+3)(\lambda+1)},
\end{equation}
where we have used the notations
\[A=\frac{d_1\big(\frac{k\pi}{L}\big)^2+1}{2},B=\frac{\lambda}{4}\Big(7d_1\big(\frac{k\pi}{L}\big)^2+1\Big)
-\frac{3(\lambda+1)}{4}\Big(d_1\big(\frac{k\pi}{L}\big)^2+1\Big)
,\]
and \[C=-\frac{\lambda(\lambda+1)}{8}\Big(12d_1\big(\frac{k\pi}{L}\big)^2+3\Big).\]

On the other hand, we know that $Q_k=d_2(\frac{k\pi}{L})^2+1+\lambda>\lambda$ and the quadratic function $AQ_k^2+BQ_k+C$ always has two roots $Q_k^0<0<Q_k^1$ since it has a positive determinant
\begin{align}\label{341}
\Delta=&\Big(\frac{\lambda}{4}\Big(7d_1\big(\frac{k\pi}{L}\big)^2+1\Big)
-\frac{3(\lambda+1)}{4}\Big(d_1\big(\frac{k\pi}{L}\big)^2+1\Big) \Big)^2\nonumber\\
&+\frac{\lambda\Big(\lambda+1\Big)\Big(d_1\big(\frac{k\pi}{L}\big)^2+1\Big)\Big(12d_1(\frac{k\pi}{L})^2+3\Big)}{4}>0.
\end{align}
To be precise, we have that
\begin{equation}\label{342}
\hat Q_k=\frac{-B+\sqrt{\Delta}}{2A}.
\end{equation}
Denoting
\begin{equation}\label{343}
\tilde Q_k=\Big(\frac{L^2}{4d_1(k\pi)^2}+1\Big)({\lambda+1}),
\end{equation}
we conclude from (\ref{340}) that
\begin{equation}\label{344}
\mathop{\rm sgn} {\mathcal{K}_3}=\mathop{\rm sgn} \Big({\frac{Q_k-\hat Q_k}{Q_k-\tilde Q_k}}\Big).\end{equation}
In light of Proposition \ref{proposition3}, we are now ready to present the following stability results in light of $\mathcal{K}_3$.
\begin{theorem}\label{theorem34}
Let $(u_k(s,x),v_k(s,x))$ be the nonconstant positive steady state of (\ref{15}) and $\Gamma_k(s)$ be the bifurcation branch of around $(\bar u,\bar v,\chi_k)$.  Then $\Gamma_k(s)$ turns to the right and $(u_k(s,x),v_k(s,x))$ is asymptotically stable if $\mathcal{K}_3>0$; $\Gamma_k(s)$ turns to the left and $(u_k(s,x),v_k(s,x))$ is unstable if $\mathcal{K}_3<0$; moreover, we have following cases:\\
(i):   if $d_1= \frac{1+\lambda}{2}(\frac{L}{k\pi})^2$, then $\mathcal{K}_3>0$ for all $d_2\in(0,\infty)$;\\
(ii):  if $d_1\in(0,\frac{1+\lambda}{2}(\frac{L}{k\pi})^2)$, then $\mathcal{K}_3>0$ if $ d_2\in(0,d_2^*)\cup (d_2^{**},\infty)$ and $\mathcal{K}_3<0$ if $d_2\in(d_2^*,d_2^{**})$ ;\\
(iii): if $d_1\in(\frac{1+\lambda}{2}(\frac{L}{k\pi})^2,\infty)$, then $\mathcal{K}_3>0$ if $ d_2\in(0,d_2^{**})\cup (d_2^*,\infty)$ and $\mathcal{K}_3<0$ if $d_2\in(d_2^{**},d_2^*)$,
where
\begin{equation}\label{345}
d_2^*=\Big(\frac{3\lambda}{2(d_1(\frac{k\pi}{L})^2+1)}-\frac{1}{4}-2\lambda+\frac{\sqrt{\Delta}}
{d_1\big(\frac{k\pi}{L}\big)^2+1}\Big)(\frac{L}{k\pi})^2,
\end{equation}
and
\begin{equation}\label{346}
d_2^{**}=\frac{\lambda+1}{4d_1}\big(\frac{L}{k\pi}\big)^4.
\end{equation}
\end{theorem}
\begin{remark}
$d_2^*$ defined in (\ref{345}) is always positive for all positive parameters since $\hat Q_k>1+\lambda$ in (\ref{342}).
\end{remark}

\begin{proof}
This follows through straightforward calculations thanks to Proposition \ref{proposition3}.
If $d_1= \frac{1+\lambda}{2}(\frac{L}{k\pi})^2$, we have that $\hat Q_k=\tilde Q_k$ and this implies that $\mathcal{K}_3>0$.

Dividing both numerator and denominator of $\frac{Q_k-\hat Q_k}{Q_k-\tilde Q_k}$ by $\big(\frac{k\pi}{L}\big)^2$, we have that
\[\mathop{\rm sgn} {\mathcal{K}_3}=\mathop{\rm sgn} \Big(\frac{d_2-d_2^*}{d_2-d_2^{**}}\Big),\]
where $d_2^*$ and $d_2^{**}$ are defined in (\ref{345}) and (\ref{346}).  Hence (\emph{ii}) and (\emph{iii}) follows from straightforward calculations.
\end{proof}

\section{Numerical simulations}\label{sec4}
In this section, we include some numerical results to illustrate the formation of patterns in (\ref{11}).  Our numerical simulations demonstrate that (\ref{11}) has positive solutions that converge to stable steady states with striking patterns such as boundary spikes and interior spikes.  In particular, we are concerned with the effect of system parameters on the formations of such patterns.  Spiky solutions can serve as realistic modelings of cell sorting or cellular concentration through chemotaxis.  Our numerics also show that small perturbation of initial data may lead to quite different structures of the steady states.

In Figure \ref{fig2} and Figure \ref{fig3}, we plot numerical simulations of the evolutions of cell densities and chemical concentration to the stable steady state that has large amplitude in contrast to the small amplitude bifurcating solutions.  We want to investigate the effect the parameters $d_1$, $d_2$ and $\chi$ on the magnitude of the spikes.  To elucidate our goals, we fix the parameters $\lambda=1$ and choose initial data to be the small perturbations from the homogeneous equilibrium $u_0=v_0=1+0.01\cos \pi x$.  The boundary spike corresponds to the phenomenon of cellular aggregation.  We see from Figure \ref{fig2} that the boundary layer is formed at $x=0$ and it quickly develops into a stable pattern.  Apparently, this structure is quite different from the small amplitude bifurcating solutions.  We see the development of multiple interior spikes in Figure \ref{fig3} by taking the domain size $L=10$.  Figure \ref{fig2} and Figure \ref{fig3} also demonstrate that the cell population density $u$ aggregates at the same location where the chemical concentration $v$ achieves its maximum.

\begin{figure}[h!]
\centering
\includegraphics[width=\textwidth,height=3in]{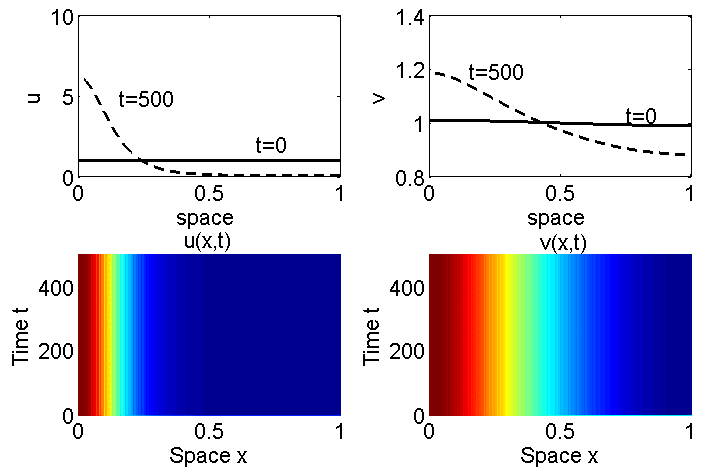}
  \caption{The formation of stable single boundary layer at $x=0$; the parameters are chosen to be $d_1=d_2=\lambda=1$ and $\chi=20$; the initial data are small perturbations from the homogeneous equilibrium $(u_0,v_0)=(\bar u,\bar v)+(0.01,0.01)\cos \pi x$.}\label{fig2}
\end{figure}

\begin{figure}[h!]
\centering
\includegraphics[width=\textwidth,height=3in]{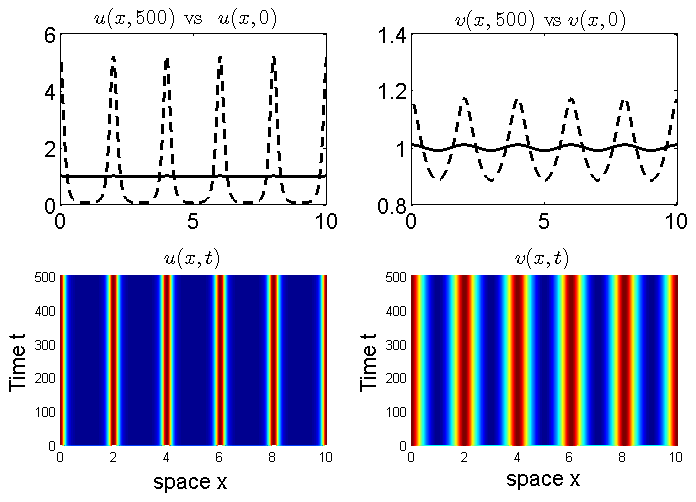}
  \caption{The formation of stable multi--spikes; all parameters and initial data are chosen to be the same as in Figure \ref{fig2} except that the domain size $L=10$.  $u(x,500)$ and $v(x,500)$ are plotted in dashed lines and the initial data $u_0$, $v_0$ in solid lines.}\label{fig3}
\end{figure}

The set of numerical solutions in Figure \ref{fig4} is included to test the effect of diffusion rates $d_1$, $d_2$ and chemoattraction rate $\chi$ on the formation of spiky solution.  Figure \ref{fig4} (a) and (b) devote to assert that small cellular and chemical diffusion rates support stable spiky solutions.  Moreover, small $d_2$ can lead to solutions with double boundary spikes when $\chi$ is not too small.  Rigorous analysis towards this was carried out on a quite similar chemotaxis model without cellular growth in \cite{W},\cite{WX} and the references therein.  According to Proposition \ref{proposition2}, the homogeneous equilibrium $(\bar u,\bar v)$ loses its stability as $\chi$ surpasses $\chi_0$, hence nonconstant positive stable steady states emerge in light of principle of exchange of stabilities.  This is numerically supported by Fig \ref{fig4} (c).  Actually, we can also perform solutions to see that $(\bar u,\bar v)$ is the global attractor of (\ref{11}) if $\chi$ is small, as shown in Proposition \ref{proposition1}.  Our numerical solutions also demonstrate that large $\frac{\chi}{d_1}$ drives the emergence of spiky solutions.

small $d_2$ can
\begin{figure}
        \centering
        \begin{subfigure}[b]{0.8\textwidth}
                \includegraphics[width=\textwidth,height=1.4in]{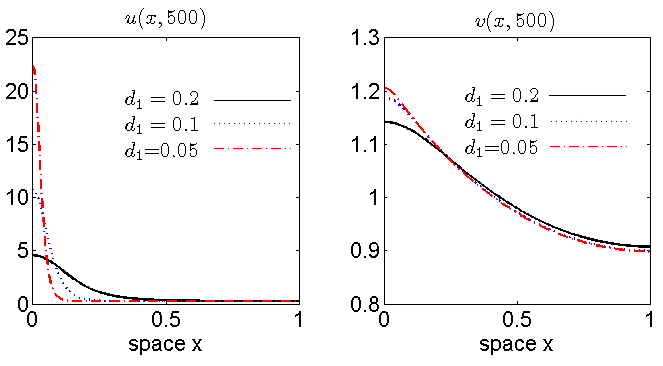}
                \caption*{(a). Stable single--boundary spikes for the cellular diffusivity $d_1$ being small.  $d_2=\lambda=1$ and $\chi=5$}
                \label{fig4a}
        \end{subfigure}\\
        \begin{subfigure}[b]{0.8\textwidth}
                \includegraphics[width=\textwidth,height=1.4in]{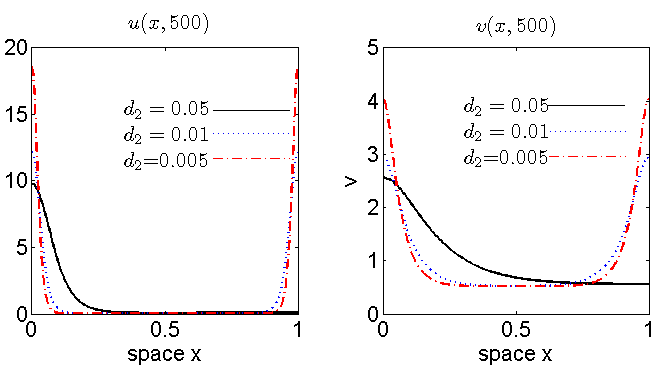}
                \caption*{(b). Stable double--boundary spikes for the chemical diffusion rate $d_2$ being small.  $d_1=\lambda=1$ and $\chi=5$}
                \label{fig4b}
        \end{subfigure}        \\
        \begin{subfigure}[b]{0.8\textwidth}
                \includegraphics[width=\textwidth,height=1.4in]{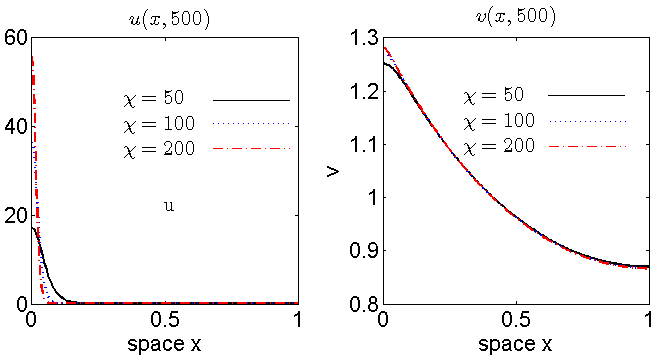}
                \caption*{(c).  Stable single--boundary spikes for the chemoattraction rate $\chi$ being small.  $d_1=d_2=\lambda=1$.}
                \label{fig4c}
        \end{subfigure}%
       \caption{Stable boundary spikes as the diffusion rates $d_1$, $d_2$ shrinks or the chemoattraction rate $\chi$ increases.  Initial data are small perturbations from the homogeneous equilibrium $(u_0,v_0)=(\bar u,\bar v)+(0.01,0.01)\cos \pi x$.  The numerics suggests that both large $\frac{\chi}{d_1}$ and small $d_2$ with properly chosen $\chi$ drive the formation of stable spikes.}\label{fig4}
\end{figure}

We present in the plot sets Figure \ref{fig5} and Figure \ref{fig6} the variations of pattern evolutions of (\ref{11}) as domain size $L$ changes.  The simulations in Figures \ref{fig5} suggest intervals with large size $L$ tend to increase the number of spikes or cell aggregates.  It is also observed that multi--spiky solutions undergoes a coarsening process, i.e., two interiors spikes merge into a single spike and new spike develops and eventually merges with another spike.  The instability of the interior spikes may be responsible for spike merging, however, in contrast to the classical Keller--Segel model where the interior spike is known to be unstable, some interior spikes are stabilized for some proper length $L$--see the stable multi--spike steady states for $L=5, 12$.  Moreover, rigourous analysis is required if one wants to find the critical threshold responsible for spike merging and stabilization.  See the analysis in section 5 of \cite{MOW} for the calculations and selection of wave modes for the single species chemotaxis model.

Figure \ref{fig4d} is included to illustrate the choice of sensitivity function on the structure of the spiky solutions.  To this end, we fix the paternosters $d_1=d_2=\lambda=L=1$ and then plot the stable steady states of (\ref{11}) with $\Phi=v$ and $\Phi=\ln v$ respectively.  The numerics in Figure \ref{fig4d} suggests that $\Phi =\ln v$ also has the saturation effect on the formation of spike height.

\begin{figure}[h!]
\centering
\includegraphics[width=3.2in,height=1.6in]{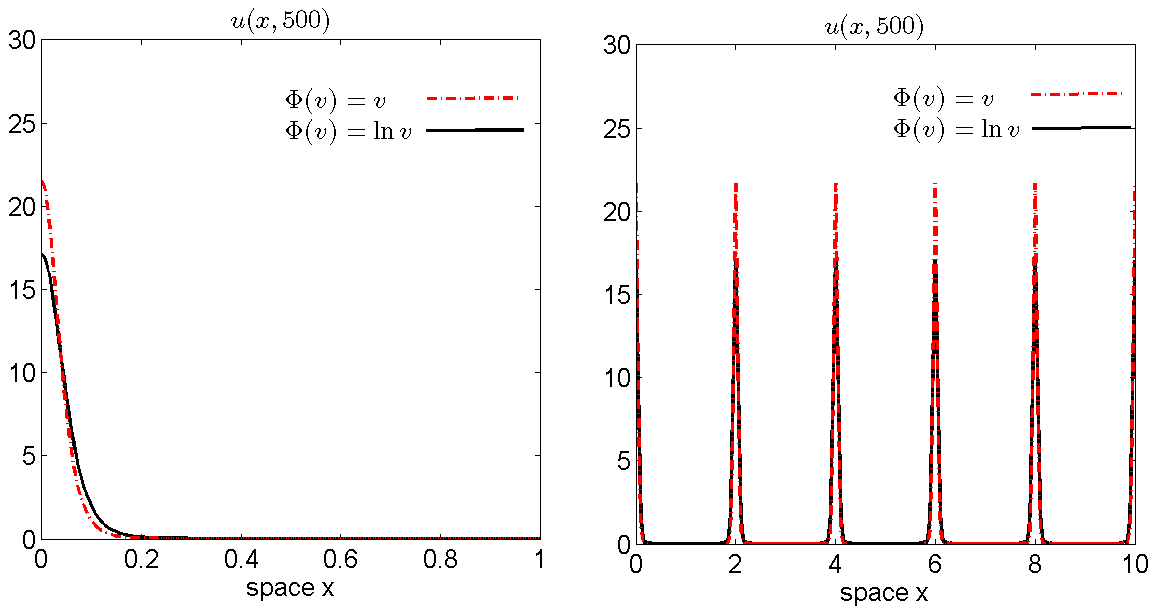}
  \caption{Spatial--temporal behaviors of $u(x,t)$ of (\ref{11}) over domains with different sizes $L$.  The parameters in the eight plots are taken to be $d_1=0.1$, $d_2=\lambda=1$, $\chi=5$; the interval of the left plot is $(0,1)$ and initial data are $u_0=2+0.01\cos 3\pi x$, $v_0=2+0.01\cos 3 \pi x$.  interval of the left plot is $(0,10)$ and initial data are $u_0=1+\cos 1.5\pi x$, $v_0=1+0.5*\cos \pi x$.  We observe that both $v$ and $\ln v$ as the sensitivity function supports spiky solutions, however, $\ln v$ tends to saturate spike height compared to $v$.}\label{fig4d}
\end{figure}

\begin{figure}[h!]
\centering
\includegraphics[width=\textwidth,height=3in]{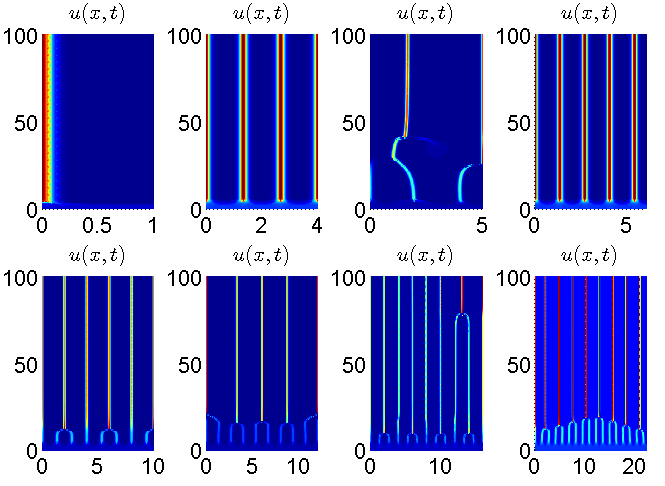}
  \caption{The formation, merging and emerging of multi--spikes; all parameters and initial data are chosen to be the same as in Figure \ref{fig2} except for the domain size. The numerics suggests that intervals with large length tends to support the formations of more spikes.}\label{fig5}
\end{figure}

\begin{figure}[h!]
\centering
\includegraphics[width=\textwidth,height=3in]{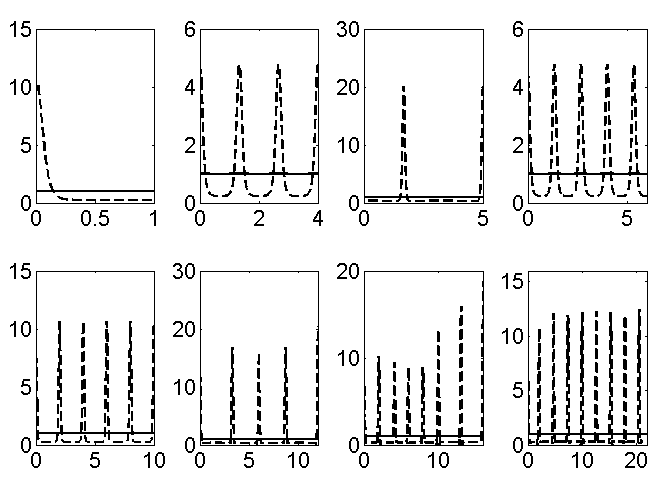}
  \caption{Spatial--temporal behaviors of $u(x,t)$ of (\ref{11}) over domains with different sizes $L$.  The parameters in the eight plots are taken to be $d_1=0.1$, $d_2=\lambda=1$, $\chi=5$; initial data are $(u_0=v_0)=(\bar u,\bar v)+(0.01,0.001,0.02)\cos 1.5\pi x$.  We observe that large interval length supports more interior spikes and the merging or emerging of spikes.}\label{fig6}
\end{figure}

\section{Conclusions and Discussions}\label{sec5}
In this paper, we establish the existence and stability of nonconstant positive steady states of one--dimensional Keller--Segel type morphogenesis models (\ref{11}) under homogeneous Neumann boundary conditions.  Its prototype system (\ref{12}) has been extensively studied by many scholars to model the cellular chemotaxis phenomenon--see the survey paper by \cite{HP}. For the simplified version (\ref{11}),  it is rigourously proved that the chemoattraction rate $\chi$ has the effect of destabilizing the homogeneous steady state.  Then spatially inhomogeneous steady states emerge from the homogeneous one through bifurcation as $\chi$ passes through a certain critical values $\chi_k$, which is explicitly given in (\ref{22}).  Our stability results in Theorem \ref{theorem33} and Theorem \ref{theorem34} suggest that the bifurcating solutions are asymptotically stable for both small and large values of the chemical diffusion rate $d_2$ and unstable if $d_2$ is of intermediate value.

We study the stability of equilibrium $(\bar u,\bar v)=(\lambda,1)$ of (\ref{11}) with respect to spatially heterogeneous perturbations.  It is shown that $\chi$ destabilizes $(\bar u,\bar v)$, which loses its stability at $\chi_0=\min_{k\in \mathbb N^+}\{\chi_k\}$ defined in (\ref{22}) in the sense of Turing's instability.  Taking $\chi$ as the bifurcation parameter, we apply the local bifurcation theories from Crandall--Rabinowitz to establish the existence of nonconstant positive solutions of the stationary system (\ref{15}).  The stability or instability of these bifurcation solutions are also obtained when the sensitivity function are chosen to be linear and logarithmic respectively.  Our main results asserts that the small amplitude bifurcating solutions are asymptotically stable if the chemical diffusion rate $d_2$ is large or small and it is unstable if $d_2$ is of intermediate value.  Numerical simulations have been performed to present the dynamical behaviors of solutions to (\ref{11}) that exhibit complex spatio--temporal patterns, such as boundary spikes and interior spikes, etc.  We also observe the coarsening process in which merging and emerging of spikes occur during the evolutions of the solutions.  Our numerical solutions suggests that interval with large size $L$ supports complex structure than that with small size.

There are a few interesting and important questions about the system (\ref{11}) that deserve exploration in the future.  First of all, one can ask for the criterion on parameters that guarantee the global existence of (\ref{11}) for $\Omega \subset \mathbb{R}^N$, $N\geq2$.  The literature suggests that blow--ups can be inhibited by the degradation in the cellular kinetics, however it remains open so far whether or not this is sufficient when $\frac{\chi}{d_1}$ is large, or the domain has a large space dimension.

The existence of non-existence of nontrivial solutions of  (\ref{15}) in multi-dimensional domain deserves exploring.  First of all, by the bifurcation theory, one can also obtain the existence of nonconstant positive steady states for $\Omega$ in multi--dimensions.  Similarly, detailed calculations can be performed to study the structures and stability of the bifurcating solutions.  Moreover, it is also important to rigourously study the steady state of (\ref{15}) with large amplitude.  For example, and construction of spatial patterns with boundary spikes, interior spike, etc. is another challenging problem that worths future attentions, even in one--dimensional domain.  Moreover, the stability of such spikes are natural and important but extremely difficult question one can work on in the future.

\end{document}